\documentclass[11pt]{article}

\usepackage[english]{babel}

\usepackage[a4paper,top=2cm,bottom=2cm,left=3cm,right=3cm,marginparwidth=1.75cm]{geometry}

\usepackage{amsmath}
\usepackage{amsthm,amssymb}
\usepackage{graphicx}
\usepackage[colorlinks=true, allcolors=blue]{hyperref}

\usepackage{amsthm}

\newtheorem{theorem}{Theorem}
\numberwithin{theorem}{section}

\newtheorem{claim}[theorem]{Claim}

\newtheorem{conjecture}[theorem]{Conjecture}
\newtheorem{question}[theorem]{Question}

\newtheorem{remark}[theorem]{Remark}
\numberwithin{equation}{section}

\title{A note on locating-dominating sets in twin-free graphs}
\author{Nicolas Bousquet\thanks{LIRIS, Université Claude Bernard Lyon 1, CNRS France. Email: \texttt{nicolas.bousquet@liris.cnrs.fr}. Research supported by ANR project GrR (ANR-18-CE40-0032).} \and Quentin Chuet\thanks{LISN, Université Paris-Saclay, France. Email: \texttt{quentin.chuet@lisn.fr}} \and Victor Falgas--Ravry\thanks{Institutionen f\"or Matematik och Matematisk Statistik, Ume{\aa} Universitet, Sweden. Email: \texttt{victor.falgas-ravry@umu.se}. Research supported by Swedish Research Council grant VR 2021-03687.} \and Amaury Jacques\thanks{\textsc{LaBRI}, Université de Bordeaux, France. Email: \texttt{amaury.jacques@labri.fr}} \and Laure Morelle\thanks{LIRMM, Université de Montpellier, France. Email: \texttt{laure.morelle@lirmm.fr}. Research supported by ANR project ELIT (ANR-20-CE48-0008-01) and French-German Collaboration ANR/DFG project UTMA (ANR-20-CE92-0027).}}

\begin{document}
	\maketitle
	
	\begin{abstract}
In this short note, we prove that every twin-free graph on $n$ vertices contains a locating-dominating set of size at most $\lceil\frac{5}{8}n\rceil$. This improves the earlier bound of $\lfloor\frac{2}{3}n\rfloor$ due to Foucaud, Henning, L\"owenstein and Sasse from 2016, and makes some progress towards the well-studied locating-dominating conjecture of Garijo, Gonz\'alez and M\'arquez.
	\end{abstract}
	
	\section{Introduction}
\subsection{Background and results}
		In this short note, we consider the problem of location-domination in graphs (see Section~\ref{section: defs} below for definitions).  Domination and location\footnote{There are several notions of location in the literature. In this paper we are exclusively concerned with neighbour-locating sets, rather than any other notion of locating sets such as e.g.\ metric-locating sets or $r$-locating sets.} are both well-studied topics in graph theory, and have been the subject of several books and surveys, see e.g.\ the recent monograph of Haynes, Hedetniemi and Henning~\cite{haynes2020topics} on domination, and the extensive bibliography of over 500 publications on identification and location problems in graphs maintained by Devin C. Jean~\cite{devinjean}.

	Locating-dominating sets, for their part, were introduced by Slater~\cite{slater1987domination,slater1988dominating} in the late 1980s, and have also received significant attention (see e.g.\ ~\cite{bousquet2023locating,colbourn1987locating,finbow1988locating,foucaud2016locating,garijo2014difference,henning2012locating,hernando2014nordhaus,slater1995locating,slater2002fault}), with motivation coming from network science and theoretical computer science. A major conjecture in the area, due to Garijo, Gonz\'alez and M\'arquez~\cite{garijo2014difference}, is that the location-domination number $LD(G)$ of an $n$-vertex twin-free graph is at most $n/2$:
	\begin{conjecture}[Location-domination conjecture]\label{conjecture: loc dom}
	If $G$ is a twin-free graph on $n$ vertices, then $LD(G)\leq \frac{n}{2}$.	
\end{conjecture}
	\noindent Given a $n$-vertex twin-free graph $G$, Garijo, Gonz\'alez and M\'arquez~\cite{garijo2014difference} showed its location-domination number $LD(G)$ can be upper-bounded by $\lfloor \frac{2n}{3}\rfloor +1$. This upper bound was improved to $\lfloor \frac{2n}{3}\rfloor$ by Foucaud, Henning, L\"owenstein and Sasse~\cite{foucaud2016locating}, who in addition showed that Conjecture~\ref{conjecture: loc dom}, if true, admits a rich family of extremal graphs. In this note, we obtain a significant improvement on these upper bounds, bringing us a step closer to Conjecture~\ref{conjecture: loc dom}:	
	\begin{theorem}\label{theorem: main}
		If $G$ is a twin-free graph on $n$ vertices, then $LD(G) \leq \lceil \frac{5n}{8} \rceil$.
	\end{theorem}
\begin{remark}\label{rem: actual result}
In fact, we prove the slightly stronger statement that a twin-free graph on $n$ vertices contains a locating set of size at most $\lfloor \frac{5n-1}{8}\rfloor$, from which Theorem~\ref{theorem: main} follows immediately. 
 \end{remark}
 \noindent Our proof of Theorem~\ref{theorem: main} leads naturally to the following intriguing question, which we believe has not appeared in the literature before 
and which ought to be of independent interest:
	\begin{question}\label{question: partition locating sets}
Let $G$ be a twin-free graph. Is there a partition $X\sqcup Y=V(G)$ of its vertex set such that $X$ and $Y$ are both locating sets?
	\end{question}
 
\subsection{Definitions and notation}\label{section: defs}
 Before we dive into the proof of Theorem~\ref{theorem: main}, we recall some key definitions. Let $G=(V,E)$ be a graph. We use standard graph theoretic notation throughout the paper.
 We denote the complement $V\setminus X$ of a subset $X\subseteq V$ by $\overline{X}$.

	Given a subset $X\subseteq V$, for any subset $Y\subseteq \overline{X}$, 
 we can define an equivalence relationship $\sim_X$ on $Y$ by letting $y\sim_Xy'$ if the neighbourhoods of $y,y'$ in $X$ are identical, $N_G(y)\cap X=N_G(y')\cap X$. We refer to the resulting partition of $Y$ into equivalence classes as the \emph{$X$-partition} of $Y$. A set of \emph{representatives} for this partition is a collection of vertices $R_Y\subseteq Y$ such that $R_Y$ meets each part of the $X$-partition of $Y$ in exactly one vertex.

Given a pair of distinct vertices $v,v'\notin X$, we say that $X$  \emph{distinguishes} $v$ from $v'$ if $v\not\sim_X v'$. This occurs precisely if there exists $x\in X$ such that $x$ sends an edge to exactly one of $v$ and $v'$. In a slight abuse of notation, we say a vertex $x$ distinguishes $v$ from $v'$ if the singleton $\{x\}$ distinguishes $v$ from $v'$. The set $X$ is said to be a \emph{locating set} if the $X$-partition of $\overline{X}$ consists of singletons only --- or, equivalently, if every pair of vertices in $\overline{X}$ can be distinguished by some vertex of $X$. Such a set is sometimes referred to as a \emph{neighbour-locating set} in the literature.

	Furthermore, a subset $X\subseteq V$ is said to be a \emph{dominating set} in $G$ if every vertex $v\in V\setminus X$ has a neighbour in $X$.  A \emph{locating-dominating set} in $G$ is then a set $X$ which is both a locating set and a dominating set. We denote by $L(G)$ and $LD(G)$ the minimum sizes of a locating and of a locating-dominating set in $G$ respectively. Note that given a locating set $X$ there is at most one vertex $v_0\notin X$ whose neighbourhood in $X$ is the empty set, so that $X\cup\{v_0\}$ is a locating-dominating set. It follows that $LD(G)\leq L(G)+1$.

	Finally, two distinct vertices $v,v'$ in $G$ are said to be \emph{twins} if their open or closed neighbourhoods are identical, i.e.\ if $N_G(v)\setminus\{v'\}=N_G(v')\setminus\{v\}$. In the context of location, one is typically interested in studying \emph{twin-free} graphs, in which there are no twin vertices.

\section{Proof of Theorem~\ref{theorem: main}}

Let $G=(V,E)$ be a twin-free graph on $n$ vertices. Since $LD(G)\leq L(G)+1$, as we remarked in Section~\ref{section: defs}, it is enough to show that $G$ contains a locating set of size a little less than $\frac{5n}{8}$. Given a set $A\subseteq V$, let $s(A)$ denote the number of 
equivalence classes of $\sim_A$ on $\overline{A}$.
Informally, this quantity $s(A)$ measures how much `separation' or `locating power' the set $A$ has. Let $S$ denote the maximum of $s(A)+s(\overline{A})$ over all sets $A\subseteq V(G)$.

We say that a 
class 
is \emph{non-trivial} if it has size at least two, and we say that set $A\subseteq V$ is a \emph{good} set if $s(A)+s(\overline{A})=S$ and $s(\overline{A})=\vert A\vert$. In other words, $A$ is a good set if it maximises $s(A)+s(\overline{A})$ and in addition every 
class of $\sim_{\overline{A}}$ on $A$ 
is trivial.  Note that $A$ is a good set implies $\overline{A}$ is a locating set.
 See Figure~\ref{fig:good} for an illustration of a good set $A$.

The following simple observation will allow us to prove the existence of good sets.


\begin{claim}\label{claim: potato thinning}
Suppose $U$ is a non-trivial 
class of $\sim_A$ on $\overline{A}$.
Then for any choice of $u\in U$, setting $U^-=U\setminus\{u\}$, we have that $s(\overline{A}\setminus U^-) \geq s({\overline{A}})$ and $s({A\cup U^-})\geq s({A})$. 
\end{claim}

\begin{proof}
Let $R_A$ and $R_{\overline{A}}$ be sets of representatives for $\sim_{\overline{A}}$ on $A$ and $\sim_A$ on $\overline{A}$, respectively, where the representative of $U$ in $R_{\overline{A}}$ is $u$.
Let $v$ and $v'$ be two vertices of $R_B$, where $B\in\{A,\overline{A}\}$.
Then there is a vertex $w\in\overline{B}$ that distinguishes $v$ from $v'$, since $v$ and $v'$ do not belong to the same $\sim_{\overline{B}}$-equivalence class.
If $B=A$ and $w\sim_{\overline{A}}u$, then $u\in \overline{A}\setminus U^-$ also distinguishes $v$ from $v'$.
Otherwise, $w$ distinguishes $v$ from $v'$ with $w\in A\subseteq A\cup U^-$ if $B=\overline{A}$, and $w\in\overline{A}\setminus U^-$ if $B=A$.
Therefore, the vertices of $R_A$ belong to different classes for $\sim_{\overline{A}\setminus U^-}$ on $A\cup U^-$, and the vertices of $R_{\overline{A}}$ belong to different classes for $\sim_{A\cup U^-}$ on $\overline{A}\setminus U^-$.
Hence the result.
\end{proof}

\begin{claim}\label{claim: there exist good sets}
Let $A\subseteq V$ be such that $s(A)+s({\overline A})=S$.
Let $R_{\overline{A}}$ be a set of representatives of $\sim_A$ on $\overline{A}$.
Then ${R_{\overline{A}}}$ is a good set.
\end{claim}
\begin{proof}
Observe that $\overline{R_{\overline{A}}}$ can be obtained from $A$ by iteratively adding for, each class $U$ of $\sim_{A}$ on ${\overline{A}}$, all but the representative vertex $u\in R_{\overline{A}}$ of $U$ to $A$.
Then, by Claim~\ref{claim: potato thinning}, we have
$s(\overline{R_{\overline{A}}})\geq s(A)$ and  $s(R_{\overline{A}})\geq s({\overline{A}})$.
By maximality of $S$, we thus have $s(\overline{R_{\overline{A}}})= s(A)$ and  $s(R_{\overline{A}})= s({\overline{A}})$.
The claim hence follows given that $s(R_{\overline{A}})+s(\overline{R_{\overline{A}}})=S$ and that $\vert R_{\overline{A}}\vert =s(A)=s(R_{\overline{A}}).$
\end{proof}

For the remainder of the proof, we fix a good set $A$.
Let $B$ be the set of all vertices that belong to a non-trivial class of $\sim_A$ on $\overline{A}$ and let $R_B$ be a set of representatives for these non-trivial classes. We denote by $k=k(A)=\vert R_B\vert$ the number of non-trivial classes in the $A$-partition of $\overline{A}$,
and assume without loss of generality that our choice of $A$ maximises the value of $k$  over all choices of a good set $A$. 
 We then let $C=\overline{A}\setminus B$ be the collection of all vertices belonging to trivial classes of $\overline{A}$.
 See Figure~\ref{fig:good} for an illustration.

 
 By Claim~\ref{claim: there exist good sets}, it is immediate that  $R_B\cup C=R_{\overline{A}}$, like $A$, is a good set. 
 In other words, $\overline{R_B\cup C}=A\cup \left(B\setminus R_B\right)$ and $\overline{A}=B\cup C$ are both locating sets. Further, let $A'$ be the vertices of $A$ that belong to a non-trivial class of $\sim_{R_B\cup C}$ on $A \cup (B \setminus R_B)$. Note that elements of $A$ distinguished by $B\cup C$ are also distinguished from each other by $R_B\cup C$. The size of $A'$ is thus at most the number of non-trivial equivalence classes in the $(R_B\cup C)$-partition of $\overline{R_B\cup C}=A\cup \left(B\setminus R_B\right)$. Since $R_B\cup C$ is a good set, it follows from the maximality of $k$ that $\vert A'\vert \leq k$.


 \begin{figure}
     \centering
     \includegraphics{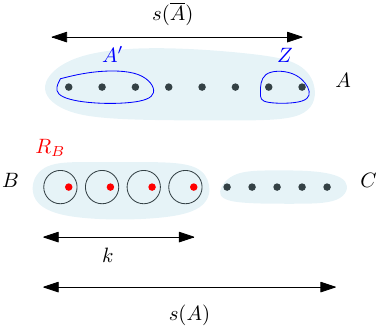}
     \caption{A good set $A$ along with the sets defined along the proof of the theorem.}
     \label{fig:good}
 \end{figure}
 
 We shall exhibit two additional locating sets.

\begin{claim}\label{claim: Acup C locating}
$A\cup C$ is a locating set.
In particular, for each pair of vertices $b,b'\in B$ such that $b\sim_A b'$, $C$ distinguishes $b$ from $b'$.
\end{claim}
		\begin{proof} This was already shown in~\cite{garijo2014difference} and~\cite{foucaud2016locating}, but we include a short proof here for completeness.

  If we have vertices $b,b'\in B$ with $b\not\sim_A b'$, then clearly these two vertices will still be distinguished by the larger set $A\cup C$. On the other hand, consider a pair of distinct vertices $b,b'\in B$ such that $b\sim_A b'$. Since $G$ is twin-free, this implies that there exists some vertex $v\in (B\setminus\{b,b'\})\cup C$ distinguishing $b$ from $b'$.

		Suppose for a contradiction that this vertex $v\in B$, i.e.\ that it belongs to a non-trivial class of the $A$-partition of $B\cup C$. Then it is easily seen that the $(A\cup\{v\})$-partition of $(B\setminus \{v\})\cup C$ contains at least one more class than the $A$-partition of $B\cup C$.

  Indeed, $v$ distinguishes $b$ from $b'$, and so splits $b$ and $b'$'s class from the $A$-partition into two distinct classes. Further, $v$'s class in the $A$-partition had size at least $2$ (or at least $3$ if it belonged to the same class as $b$ and $b'$), whence the removal of $v$ from $B\cup C$ does not turn any class from the $A$-partition into the empty class.

  By Claim~\ref{claim: potato thinning}, this implies  that $s(A\cup \{v\})+s(\overline{A}\setminus \{v\})> s(A)+s(\overline{A})=S$, contradicting the maximality of $S$. Thus we must have that $v\in C$, and in particular that $A\cup\{v\}\subseteq A\cup C$ distinguishes $b$ from $b'$.

	Thus every pair of distinct vertices $b,b'\in B$ are distinguished by $A\cup C$, and $A\cup C$ is a locating set, as desired.
		\end{proof}
		We now construct a final locating set. We first prove that there is a small subset $Z \subseteq A$ such that the $Z$-partition of $B$ is equal to the $A$-partition of $B$.
		
\begin{claim}\label{claim: separate k classes with at most k-1 vertices}
There exists a subset $Z \subseteq A$ of at most $k-1$ vertices such that the $Z$-partition of $B$ has the same $k$ classes as the $A$-partition of $B$.
\end{claim}
		\begin{proof}
		We construct the set $Z$ greedily. Initially $Z=\emptyset$, and there are $k_Z=1$ classes in the $Z$-partition of $B$.
  Note that for any $Z\subseteq A$, each class in the $Z$-partition of $B$ is the union of classe in the $A$-partition of $B$.

		Suppose $k_Z<k$. Then there are two distinct parts $U$ and $U'$ in the $A$-partition of $B$ that lie inside the same part of the $Z$-partition of $B$. Since $A$ distinguishes $U$ from $U'$, this implies that  there exists $a\in A\setminus Z$ such that $a$ sends edges to all vertices in one of the parts $U$, $U'$, and no edges to the vertices in the other part. Adding $a$ to $Z$ increases the size of $Z$ by $1$ and the size of $k_Z$ by at least $1$.
		
		Repeating this procedure at most $k-1$ times, we construct the desired set $Z$ with $\vert Z\vert\leq k-1$, such that $k_Z=k$, and the $Z$-partition and $A$-partition of $B$ coincide.
		\end{proof}
  


		\begin{claim}\label{claim: union o A, B', C and Z locating}
			$W:=A' \cup Z\cup R_B \cup C$ is a locating set.
		\end{claim}
		\begin{proof}
			We must show that $W$ can distinguish all pairs of vertices from $\overline{W}=\left(A\setminus (A'\cup Z)\right)\cup (B\setminus R_B)$.
			
			Consider a pair of vertices $a,a'\in A\setminus (A'\cup Z)$. Since $B\cup C$ is a locating set for $A$, there exists a class $U$ in the $A$-partition of $B\cup C$ sending edges to exactly one of $a,a'$. Since $R_B\cup C$ contains one element from each of the $\sim_A$ equivalence classes, there exists $u \in U\cap (R_B\cup C)$ which distinguishes $a$ from $a'$.

			
			Consider now a pair of vertices $b,b'\in B\setminus R_B$. If they belong to the same non-trivial class in the $A$-partition of $B\cup C$, then, by Claim~\ref{claim: Acup C locating}, $C \subseteq W$ distinguishes $b$ from $b'$. On the other hand, if $b,b'$ belong to different non-trivial parts of the $A$-partition of $B\cup C$, then they must be distinguished by $Z\subseteq W$.
			
			 Finally, consider vertices $a\in A\setminus (A'\cup Z)$ and $b\in B\setminus R_B$. By construction of $A'$, all elements of $A\setminus A'$ lie in trivial parts of the $(R_B\cup C)$-partition of $A\cup \left(B\setminus R_B\right)$. In particular, $W$ contains the set $R_B\cup C$, which distinguishes $a$ from $b$.
		\end{proof}
We have thus identified four locating sets, giving us four different upper bounds on $L(G)$. Setting $b:= \vert B\vert$ and $c:=\vert C\vert$, we have:
\begin{align}
	L(G)& \leq \vert A\cup \left(B\setminus R_B\right)\vert = (n-b-c) +(b-k)= n-c-k \label{eq1}\\
	L(G)&\leq  \vert B\cup C\vert = b+c \label{eq2}\\
	L(G)&\leq  \vert A\cup C\vert = (n-b-c)+c=n-b \label{eq3}\\
	L(G)&\leq  \vert A'\cup R_B\cup C\cup Z\vert \leq k+ k + c+ k-1=c+3k-1 \label{eq4}.
\end{align}
Adding~\eqref{eq1} and~\eqref{eq4}, we see that 
\begin{align*}
2 L(G) \leq n +2k-1,
\end{align*}
which, if $k\le\frac{n+3}{8}$, implies that
\begin{align*}
L(G) \leq \frac{5n-1}{8},
\end{align*}
while adding~\eqref{eq1}, \eqref{eq2} and~\eqref{eq3} we obtain
\begin{align*}
	3 L(G) \leq 2n -k,
\end{align*}
which, if $k\ge\frac{n+3}{8}$, implies that
\begin{align*}
L(G) \leq \frac{5n-1}{8}.
\end{align*}
Hence there exists a locating set in $G$ of size at most $\lfloor \frac{5n-1}{8}\rfloor$. Adding at most one vertex to our locating set, we form a locating-dominating set of size at most $\lfloor\frac{5n+7}{8}\rfloor\le\lceil \frac{5n}{8}\rceil$. This concludes the proof of Theorem~\ref{theorem: main}.\qedsymbol
	
\section{Concluding remarks}
	Our proof of Theorem~\ref{theorem: main}, which was inspired by earlier arguments of Foucaud, Hening, L\"owenstein and Sasse~\cite{foucaud2016locating}, relies on maximising the sum of the `locating power' of a set $A$ and its complement $\overline{A}$. It is natural to ask about the value of this sum in general.
	\begin{question}\label{question: maximising locating powers over bipartitions}
		How small can $\max_{A\subseteq V(G)} s_A +s_{\overline{A}}$ be for a twin-free graph $G$ on  $n$ vertices?
	\end{question}
	\noindent As can be seen from our Question~\ref{question: partition locating sets} in the introduction, for all we know the answer to the question above could be $n$, and every twin-free graph could have a partition of its vertex sets into mutually locating sets. If true, this would be a very strong result, and bring us very close to a proof of Conjecture~\ref{conjecture: loc dom}.

	More generally, one could ask for the maximum locating power over $k$-partitions of a vertex set. Given a partition of the vertex set of a graph $G$ into $k$-set $\mathbf{A}=\{A_1, A_2, \ldots, A_k\}$, set $s_{\mathbf{A}}:=\sum_{i=1}^k s_{A_i}$. Further, let $s_k(G)$ denote the maximum of $s_{\mathbf{A}}$ over all $k$-partitions $V(G)=\sqcup_{i=1}^k A_i$ of $V(G)$, and let $s_k(n)$ denote the minimum of $s_k(G)$ over all $n$-vertex twin-free graphs $G$.
		\begin{question}\label{question: maximising locating powers over k-partitions}
What is the value of $s_k(n)$?
		\end{question}
	\noindent It is easily seen that $s_k(n)\leq (k-1) n$, and that $s_{n}(n)=2n-1$. It would be particularly interesting to determine the order of $\max_{2\leq k\leq n}s_k(n)$ --- is it linear or superlinear in $n$?
	
	\subsection*{Acknowledgements}
	The authors would like to express their gratitude to the organisers of the \emph{Journ\'ees Graphes et Algorithmes 2023} (JGA 23) in Lyon, and in particular for setting up this working group. In  addition, we would like to thank Florent Foucaud for an inspiring lecture at the JGA, as well as Marthe Bonamy~\cite{marthe} and Fr\'ed\'eric Havet~\cite{fred} for letting us know about upcoming work on Conjecture~\ref{conjecture: loc dom} from a different angle.

\end{document}